\documentclass{amsart}

\usepackage{amssymb}
\newcommand{\eps}{{\varepsilon}}

\newcommand{\Z}{{\mathbb Z}}
\newcommand{\Q}{{\mathbb Q}}

\newtheorem{corollary}{Corollary}

\title{Rational Quartic Reciprocity II}
\author{Franz Lemmermeyer}
\address{M\"orikeweg 1, 73489 Jagstzell, Germany}
\email{hb3@ix.urz.uni-heidelberg.de}
\subjclass{11 R 16, 11 A 15}

\begin{document}
\maketitle

\section{Introduction}

Let $m = p_1 \cdots p_r$ be a product of primes $p_i \equiv 1 \bmod 4$
and assume that there are integers $A, B, C \in \Z$ such that 
$A^2 = m(B^2+C^2)$ and $A-1 \equiv B \equiv 0 \bmod 2$, 
$A+B\equiv 1 \bmod 4$. Then 
\begin{equation}\label{EC}
	\Big(\frac{A+B\sqrt{m}}{p}\Big) = \Big(\frac{p}{m}\Big)_4
\end{equation}
for every prime $p \equiv 1 \bmod 4$ such that
$(p/p_j) = +1$ for all $1 \le j \le r$. 
This is 'the extension to composite values of $m$' that 
was referred to in \cite{LAA}, to which this paper is an
addition. Here I will fill in the
details of a proof, on the one hand because I was
requested to do so, and on the other hand because this
general law can be used to derive general versions of 
Burde's and Scholz's reciprocity laws.

Below I will sketch an elementary proof of \eqref{EC} 
using induction built on the results of \cite{LAA}, and 
then use the description of abelian fields by characters
to give a direct proof.

\section{Proof by Induction}

Using induction over the number of prime factors of $m$
we may assume that (\ref{EC}) is true if $m$ has
$r$ different prime factors. 

Now assume that $m = p_1 \cdot m'$; we choose integers
$A, B, A_1, B_1$ such that $B$ and $B_1$ are even and
$A+B \equiv A_1+B_1 \equiv 1 \bmod 4$, satisfying
$$\begin{array}{rclrcl}
 A^2  &=& m(B^2+C^2),        & A_1^2 &=& p_1(B_1^2+C_1^2),
	\qquad \text{ and put} \\
\alpha   &=& A+B\sqrt{m},  & \alpha_1 &=& A_1 + B_1 \sqrt{p_1}.
\end{array} $$
Then $K = \Q(\sqrt{\alpha}\,)$ and $K_1 = \Q(\sqrt{\alpha_1}\,)$
are cyclic quartic extensions of conductors $m$ and $p_1$,
respectively. 

Consider the compositum $K_1K$; it is an abelian extension of type
$(4,4)$ over $\Q$, and it clearly contains
$F =  \Q(\sqrt{m'}, \sqrt{p_1}\,)$. Moreover, $F$
has three quadratic extensions in $K_1K$, namely
$F(\sqrt{\alpha}\,)$, $F(\sqrt{\alpha_1}\,)$,  and
$L = F(\sqrt{\alpha\alpha_1}\,)$. It is not hard to see
that $L$ is the compositum of a cyclic quartic extension
$\Q(\sqrt{\alpha'}\,)$ of conductor $m'$ and 
$\Q(\sqrt{p_1}\,)$. Since $\alpha\alpha_1$
and $\alpha'$ differ at most by a square in $F$, we find
$$\Big(\frac{\alpha'}{p}\Big) \ = \
	\Big(\frac{\alpha}{p}\Big)\Big(\frac{\alpha_1}{p}\Big).$$
On the other hand, by the induction hypothesis we have
$(\alpha'/p) = (p/m')_4$, hence we find
$$\Big(\frac{\alpha}{p}\Big) \ = \
	\Big(\frac{\alpha'}{p}\Big) \Big(\frac{\alpha_1}{p}\Big) \ = \
	\Big(\frac{p}{m'}\Big)_4 \Big(\frac{p}{p_1}\Big)_4  \ = \
	\Big(\frac{p}{m}\Big)_4.$$
This is what we wanted to prove.

\section{Proof via Characters}

Let $K$ be a cyclotomic field with conductor $f$. Then
it is well known (see \cite{Was} for the necessary background)
that the subfields of $\Q(\zeta_f)$ correspond biuniquely to a 
subgroup of the character group of $(\Z/f\Z)^\times$. 

Let $m = p_1 \ldots p_r$ be a product of primes $p \equiv 1 \bmod 4$,
and let $\phi_j$ denote the quadratic character modulo $p_j$.
There exist two quartic characters modulo $p_j$, namely
$\omega_j$ (say) and $\omega_j^{-1} = \phi_j\omega_j$;
for primes $p$ such that $\chi_j(p) = (p/p_j) = +1$ we have
$\omega_j(p) = (p/p_j)_4$.

The quadratic subfield $\Q(\sqrt{m}\,)$ of $L = \Q(\zeta_m)$
corresponds to the subgroup $\langle \phi \rangle$, where
$\phi = \phi_1 \cdots \phi_r$; similarly, there is
a cyclic quartic extension $K$ contained in $L$ which
corresponds to $\langle \omega \rangle$, where $\omega$
is a character of order $4$ and conductor $m$. Moreover,
$K$ contains $\Q(\sqrt{m}\,)$, hence we must have 
$\omega^2 = \phi$. This implies at once that
$\omega = \omega_1 \cdots \omega_r \cdot \phi'$, where
$\phi'$ is a suitably chosen quadratic character.
By the decomposition law in abelian extensions a
prime $p$ splitting in $\Q(\sqrt{m}\,)$ will split
completely in $K$ if and only if $\omega(p) = +1$,
i.e. if and only if $(p/m)_4 = + 1$ (the quadratic
character $\phi'$ does not influence the splitting
of $p$ since $\phi'(p) = 1$).

By comparing this with the decomposition law in Kummer
extensions we see immediately that Eq. \eqref{EC} holds.

\noindent {\bf Remark.} If we define $(p/2)_4 = (-1)^{(p-1)/8}$
for all primes $p \equiv 1 \bmod 8$, then the above proofs
show that \eqref{EC} is also valid for even $m$; one
simply has to replace the cyclic quartic extension of conductor
$p$ by the totally real cyclic quartic extension of conductor $8$,
i.e. the real quartic subfield of $\Q(\zeta_{16})$.

\section{Some Rational Quartic Reciprocity Laws}
\subsection*{Burde's Reciprocity Law}
Let $m$ and $n$ be coprime integers, and assume that
$m = \prod p_i$ and $n = \prod q_j$ are products of
primes $\equiv 1 \bmod 4$. Assume moreover that $(m/q_j)
= (n/p_i) = +1$ for all $p_i$ and $q_j$. Write
$m = a^2 + b^2$, $n = c^2+d^2$ with $ac$ odd; then
we can prove as in \cite{LAA} that
$$\Big(\frac{m}{n}\Big)_4 \Big(\frac{n}{m}\Big)_4
	= \Big(\frac{ac-bd}{m}\Big) 
	= \Big(\frac{ac-bd}{n}\Big).$$

\noindent {\bf Remark.} It is easy to deduce Gauss' criterion
for the biquadratic character of $2$ from Burde's law. In fact,
assume that $p = a^2 + 16b^2 \equiv 1 \bmod 8$ is prime, and
choose the sign of $a$ in such a way that $a \equiv 1 \bmod 4$; 
then
$$\Big(\frac{2}{p}\Big)_4 \Big(\frac{p}{2}\Big)_4 =
	\Big(\frac{a-4b}{2}\Big) = \Big(\frac{2}{a-4b}\Big).$$
Since $(p/2) = (-1)^{(p-1)/8}$ and $p-1 = a^2 - 1 + 16b^2
\equiv (a-1)(a+1) \bmod 16$ we find
$\frac{p-1}8 = \frac{a-1}4 \frac{a+1}2 \equiv \frac{a-1}4 \bmod 2$,
and this gives $(-1)^{(p-1)/8} = (2/a)$. Thus
$(2/p)_4 = (2/a)(2/a+4b) = (2/a^2+4b) = (2/1+4b) = (-1)^b.$

\subsection*{Scholz's Reciprocity Law}

Let $\eps_m = t+u \sqrt{m}$ be a unit in $\Q(\sqrt{m}\,)$ 
with norm $-1$. Putting $\eps_m \sqrt{m} = A+B\sqrt{m}$
we find immediately
\begin{equation}\label{SRL}
\Big(\frac{\eps_m}{p}\Big) = 
	\Big(\frac{m}{p}\Big)_4 \Big(\frac{p}{m}\Big)_4 
\end{equation}
for all primes $p \equiv 1 \bmod 4$ such that $(p_j/p) = 1$
for all $p_j \mid m$. If $n$ is a product of such primes $p$,
this implies
$$\Big(\frac{\eps_m}{n}\Big) = 
	\Big(\frac{m}{n}\Big)_4 \Big(\frac{n}{m}\Big)_4. $$
Moreover, if the fundamental unit of $\Q(\sqrt{n}\,)$ has
negative norm, we conclude that
$$\Big(\frac{\eps_m}{n}\Big) = 
	\Big(\frac{m}{n}\Big)_4 \Big(\frac{n}{m}\Big)_4
	= \Big(\frac{\eps_n}{m}\Big).$$

The general version of Scholz's reciprocity law has a few
nice corollaries:

\begin{corollary}
Let $m$ and $n$ satisfy the conditions above, and suppose that $m = rs$;
assume moreover that the fundamental units $\eps_r$ and $\eps_s$
of $\Q(\sqrt{r}\,)$ and $\Q(\sqrt{s}\,)$ have negative norm. Then 
$$\Big(\frac{\eps_m}{n}\Big) = 
	 \Big(\frac{\eps_r}{n}\Big)\Big(\frac{\eps_s}{n}\Big).$$
\end{corollary}

\begin{proof}
This is a simple computation:
$$ \Big(\frac{\eps_m}{n}\Big) = 
	\Big(\frac{m}{n}\Big)_4 \Big(\frac{n}{m}\Big)_4 = 
	\Big(\frac{r}{n}\Big)_4 \Big(\frac{n}{r}\Big)_4
	\Big(\frac{s}{n}\Big)_4 \Big(\frac{n}{s}\Big)_4 =
	\Big(\frac{\eps_r}{n}\Big)\Big(\frac{\eps_s}{n}\Big),$$
where we have twice applied \eqref{SRL}.
\end{proof}

\begin{corollary}
Let $m = p_1 \cdots p_t$ and $n$ satisfy the conditions above; then
$$\Big(\frac{\eps_m}{n}\Big) = 
	\Big(\frac{\eps_1}{n}\Big) \cdots
		\Big(\frac{\eps_t}{n}\Big), $$
where $\eps_j$ denotes the fundamental unit in $\Q(\sqrt{p_j}\,)$.
\end{corollary}
This is a result due to Furuta \cite{Fur}; its proof is clear.

\section{Some Remarks on the $4$-rank of class groups}
The reciprocity laws given above are connected with the
$4$-rank of class groups: 
let $k$ be a real quadratic number field of discrimniant $d$,
and assume that $d$ can be written as a sum of two squares.
It is well known (\cite{Sch}) that the quadratic unramified 
extensions of $k$ correspond to factorizations $d = d_1d_2$
of $d$ into two relatively prime discriminants $d_1, d_2$
with at least one of the $d_i$ positive,
and that cyclic quartic extensions which are unramified
outside $\infty$ correspond to $C_4$-extensions
$d = d_1 \cdot d_2$, where $(d_1/p_2) = (d_2/p_1) = +1$
for all primes $p_j \mid d_j$. 

Let $K=k(\sqrt{\alpha}\,)$ be such an extension, corresponding 
to $d = d_1 \cdot d_2$. Then any quartic cyclic extension of
$k$ which contains $\Q(\sqrt{d_1},\sqrt{d_2}\,)$ and
which is unramified outside $\infty$ has the form  
$K' = k(\sqrt{d'\alpha})$, where $d'$ is a product of prime
discriminants occuring in the factorization of $d$ as a
product of prime discriminants. Since these prime discriminants
are all positive, either all of these extensions
$K'/k$ are totally real, or all of them are totally complex.
Scholz \cite{Sch} has sketched a proof for the fact that
the $K'$ are totally real if and only if $(d_1/d_2)_4 = (d_2/d_1)_4$;
an elementary proof was given in \cite{Lem}.

In addition to the references given in \cite{LAA} we should
remark that Kaplan \cite{Kap} has also proved the general
version of Burde's reciprocity law and noticed the 
connection with the structure of the $2$-class groups of
real quadratic number fields.


\begin{thebibliography}{999}

\bibitem{Fur} Y. Furuta,
{\em Norms of units of quadratic fields}, 
J. Math. Soc. Japan  {\bf  11} (1959), 139--145
%

\bibitem{Kap} P. Kaplan,
{\em Sur le $2$-groupe des classes d'id\'{e}aux des
    corps quadratiques},
J. Reine Angew. Math. {\bf 283/284} (1974), 313--363
%

\bibitem{LAA} F. Lemmermeyer,
{\em Rational quartic re\-ci\-pro\-ci\-ty}, 
Acta Arithmetica {\bf 67} (1994), 387--390
%

\bibitem{Lem} F.~Lemmermeyer,
{\em The $4$-class group of real quadratic number fields},
submitted
%

\bibitem{Sch} A. Scholz,
{\em \"Uber die L\"osbarkeit der Gleichung $t^2-Du^2 =-4$},  
Math. Z. {\bf  39} (1934), 95--111.
%

\bibitem{Was} L.~Washington,
{\em Introduction to Cyclotomic Fields}, Graduate Texts in Mathematics
{\bf 83}, Springer Verlag 1982
%

\end{thebibliography}
\end{document}